\documentclass{amsart}

\usepackage[utf8]{inputenc}

\usepackage{amsmath}
\usepackage{amssymb}
\usepackage{amsfonts}
\usepackage{amsthm}

\usepackage{hyperref}

\usepackage{color}

\newtheorem{theorem}{Theorem}

\newtheorem{lemma}[theorem]{Lemma}
\newtheorem{proposition}[theorem]{Proposition}
\newtheorem{corollary}[theorem]{Corollary}
\newtheorem{problem}[theorem]{Problem}
\theoremstyle{definition}
\newtheorem{definition}[theorem]{Definition}

\def\R{\mathbb{R}}
\def\C{\mathbb{C}}
\def\N{\mathbb{N}}
\def\Z{\mathbb{Z}}
\def\Q{\mathbb{Q}}

\def\al{\alpha}

\def\la{\lambda}

\def\dd{\mathcal{D}}

\def\ll{\mathcal{L}}
\def\mm{\mathcal{M}}
\def\nn{\mathcal{N}}

\newcommand{\abs}[1]{\left|#1\right|}
\newcommand{\bo}[1]{\mathbf{#1}}

\def\boe{\bo{e}}
\def\bd{\bo{d}}

\def\bv{\bo{v}}
\def\bx{\bo{x}}
\def\by{\bo{y}}
\def\bz{\bo{z}}
\def\bw{\bo{w}}

\def\bnull{\bo{0}}

\title[Rational Matrices that admit finite digit representations]{Characterization of rational matrices that admit finite digit representations}

\author{J.~Jankauskas and J.~M.~Thuswaldner}
\address{Mathematik und Statistik, Montanuniversit\"at Leoben, Franz Josef Stra\ss{}e 18, A-8700 Leoben, Austria}
\email{jonas.jankauskas@gmail.com}
\email{joerg.thuswaldner@unileoben.ac.at}

\thanks{The post-doctoral position of the first author is supported by the Austrian Science Fund (FWF) project M2259 \emph{Digit Systems, Spectra and Rational Tiles} under the Lise Meitner Program. The second author is supported by FWF project P29910 \emph{Dynamics, Geometry and Arithmetics in Number Representations}.}


\begin{document}

\begin{abstract}
Let $A$ be an $n \times n$ matrix with rational entries and let  
\[
\Z^n[A] := \bigcup_{k=1}^{\infty} \left( \Z^n + A\Z^n + \dots + A^{k-1}\Z^n\right)
\] 
be the minimal $A$-invariant $\Z$-module containing the lattice $\Z^n$. If $\dd\subset\Z^n[A]$ is a finite set we call the pair $(A,\dd)$ \emph{a  digit system}. We say that $(A,\dd)$ has  \emph{the finiteness property} if each $\bz \in \Z^n[A]$ can be written in the form
\[
\bz = \bd_0 + A\bd_1 + \dots + A^k\bd_k,
\]
with $k\in\N$ and \emph{digits} $\bd_j \in \dd$ for $0\le j\le k$. We prove that for a given matrix $A \in M_n(\Q)$ there is a finite set  $\dd\subset\Z^n[A]$ such that $(A, \dd)$ has the finiteness property if and only if $A$ has no eigenvalue of absolute value $< 1$. This result is the matrix analogue of \emph{the height reducing property} of algebraic numbers. In proving this result we also characterize integer polynomials $P \in \Z[x]$ that admit digit systems having the finiteness property in the quotient ring $\Z[x]/(P)$.
\end{abstract}

\maketitle

Let $A$ be an $n\times n$ integer matrix and let $\dd \subset \Z^n$ be finite. The pair 
$(A, \dd)$ is called \emph{a  digit system} in the lattice $\mathbb{Z}^n$. The matrix $A$ is called \emph{the base} of the digit system, the set $\dd$ is called its \emph{digit set}. The pair $(A, \dd)$ is said to have \emph{the finiteness property}, if every vector $\bz \in \Z^n$ can be written as a finite sum using only the digits from $\dd$, multiplied by non-negative powers of $A$, {\it i.e.}, if every vector $\bz \in \Z^n$ admits \emph{a radix representation}
\begin{equation}\label{expr}
\bz = \bd_0 + A\bd_1 + \dots + A^k\bd_k,
\end{equation}
with \emph{digits} $\bd_j \in \dd$ for $0\le j\le k$. Such representations, in general, need not be unique. If they are, we say that $(A,\dd)$ has \emph{the  unique representation property}. If one defines the set $\dd[A] \subset \R^n$ by
\begin{equation*}
\dd[A] := \{\bd_0 + A\bd_1 + \dots + A^k\bd_k \in \R^n, \bd_j \in \dd, 0 \leq j \leq k, \, j, k \in \Z\},
\end{equation*}
then the finiteness property of the pair $(A, \dd)$ can be restated as $\Z^n = \dd[A]$. 

For a ring $R\subset \mathbb{R}$ in all what follows we will denote $M_n(R)$ the set of $n\times n$ matrices with entries taken from $R$. Recall that a matrix of $M_n(R)$ is called \emph{expanding}, if each of its eigenvalues is strictly greater than $1$ in absolute value. In 1993 Vince \cite{Vin1, Vin2}  demonstrated that for each expanding integer matrix $A \in M_n(\Z)$, there exists a finite digit set $\dd$, consisting of integer vectors, such that $(A, \dd)$ is a digit system in $\Z^n$ with finiteness property (this result is essentially contained in \cite[Lemma~2]{Vin2}, although it was not stated in his paper explicitly in this form).  The basic underlying principle behind the finiteness property is the ultimate periodicity of the mapping $\Phi: \Z^n \to \Z^n$, $\Phi(\bx)=A^{-1}(\bx - \bd(\bx))$, where $\bd(\bx) \in \dd$ satisfies $\bx \equiv \bd \pmod{A\Z^n}$. Vince also noted that, when $A$ has at least one eigenvalue of absolute value $< 1$, then $(A, \dd)$ cannot have the finiteness property. Moreover, in \cite[Proposition~4]{Vin2} he showed that a digit system cannot possess the unique representation property unless $A$ is expanding. We refer to A.~Kov\'acs~\cite{KoLa}, where problems on this topic are formulated. Although we attribute the matrix version formulation to Vince, all basic principles were understood much earlier in the context of  number systems defined in orders of number fields.  With special emphasis on the unique representation property, these number systems were studied extensively by K{\'a}tai and Szab{\'o} \cite{KaSz}, K{\'a}tai and B.~Kov{\'a}cs \cite{KatKov1, KatKov2, KovB}, Gilbert \cite{Gi2, Gi8}, B.~Kov{\'a}cs and Peth\H{o} \cite{KovPet1, KovPet2}, Burcsi and A.~Kov\'acs~\cite{BK:08}, Akiyama and Rao~\cite{AkRa}, Scheicher~\cite{Sch}, and many others. More recently, the set of algebraic numbers $\al$ that admit number systems in $\Z[\al]$ with finiteness property was investigated and fully characterized in the series of papers \cite{AkDrJa, AkThZa2, AkThZa, AkiZai} by Akiyama and his co-authors. In this context the finiteness property is also known as \emph{the height reducing property} of the minimal polynomial of $\alpha$. We mention that the characterization of the unique representation property is far from being complete. 

In the present note, we extend Vince's results on the finiteness property in two directions: firstly, we deal with cases when $A$ has rational (not necessarily integer) entries; secondly, we deal with the situation when $A$ has eigenvalues $\abs{\la}=1$. 

Let $A\in M_n(\mathbb{Q})$. Notice that the lattice $\Z^n$ is no longer $A$-invariant when $A$ has non-integer entries. Thus for a rational matrix $A$ we cannot expect the representations of the form \eqref{expr} to lie in $\Z^n$. We make up for this by introducing the $A$-invariant $\Z$-module 
\begin{equation}\label{defZA}
\Z^n[A] := \bigcup_{k=1}^{\infty} \left( \Z^n + A\Z^n + \dots + A^{k-1}\Z^n\right)
\end{equation}
which turns out to be the natural space containing the digit representations in this setting. Clearly, $\Z^n[A]$ is the smallest $A$-invariant $\Z$-module that contains the lattice $\Z^n$. As the example $\Z[3/2]$ shows, in general the module $\Z^n[A]$ can no longer be regarded as a lattice in $\R^n$ when the entries of $A$ are not integers. Using the module $\Z^n[A]$ we now give the exact definition of the objects we are interested in.

\begin{definition}\label{def:rationalNS}
Let $A\in M_n(\mathbb{Q})$ and let $\dd\subset\Z^n[A]$ be finite. Then the pair $(A,\dd)$ is called \emph{a digit system} in $\Z^n[A]$ with \emph{base} $A$ and \emph{digit set} $\dd$. If $\dd[A]=\Z^n[A]$, {\it i.e.}, if each 
$\bz \in \Z^n[A]$ admits a finite \emph{radix representation}
\begin{equation}\label{expr2}
\bz = \bd_0 + A\bd_1 + \dots + A^k\bd_k,
\end{equation}
with digits $\bd_j \in \dd$ for $0\le j\le k$ we say that $(A, \dd)$ has \emph{the  finiteness property}. 
If each $\bz\in\Z^n[A]$ has a unique representation of the form \eqref{expr2} then $(A,\dd)$ is said to possess \emph{the unique representation property}.
\end{definition}

Using this definition we are able to state our main result.

\begin{theorem}\label{thmMain}
Let $A$ be an $n \times n$ matrix with rational entries. There is a digit set $\dd \subset \Z^n[A]$ that makes $(A, \dd)$ a digit system in $\Z^n[A]$ with finiteness property if and only if $A$ has no eigenvalue $\la$ with $\abs{\la} < 1$. The digit set $\dd$ can even be chosen to be a subset of $\Z^n$.
\end{theorem}

While Vince \cite{Vin2} derives his finiteness result by recasting the argument of the ultimate periodicity from the digit systems \cite{Gi2, Gi8, KatKov1, KatKov2,  KaSz, KovB, KovPet1} in the matrix form, our case is more subtle. The reason for this is the existence of infinite orbits of points $\bx \in \Z^n$ under the action of $A^{-1}$, when the matrix $A$ possesses Jordan blocks of the orders $\geq 2$ corresponding to eigenvalues $\la$ with $\abs{\la}=1$.  This can be easily seen, for instance, by taking $A$ to be the shear matrix:
\[
A = \begin{pmatrix}
1 & 1\\
0 & 1
\end{pmatrix}.
\]
Note that for $\bx=(x_1, x_2)^T$, $A^{-n}\bx = (x_1-nx_2, x_2)$, so that $\Phi^{n}(\bx)$ might ultimately diverge if $x_2\ne 0$.

We deal with this by decomposing $\Z^n[A]$ into simpler submodules and constructing digit systems with finiteness property there first. We are going to prove Theorem~\ref{thmMain} by building on a result proved by Akiyama~{\it et al.}~\cite{AkThZa} for number systems in orders of number fields. Before stating it in the form that is convenient for our applications, let us digress into some definitions and notation from polynomial rings.

Let $P \in \Z[x]$ be a polynomial in a single variable $x$ with integer coefficients. Similarly to \eqref{expr}, for a finite subset $\nn \subset \Z$, we define
\begin{equation*}
\nn[x] :=\{ d_0+d_1x+\dots+ d_k x^k, d_j \in \nn, 0 \leq j \leq k, \, j, k \in \Z\}, 
\end{equation*} 
and call the pair $(P, \nn)$ \emph{a digit system} in the ring $\Z[x]/(P)$ with \emph{a digit set} $\nn$. We say that $(P,\nn)$ has \emph{the finiteness property} if 
\[
\Z[x]=\nn[x] + (P),
\]
{\it i.e.}, if each residue class of $\Z[x]/(P)$ has a representative in $\nn[x]$.

With all this in mind, we can now recast the main result of \cite{AkThZa} in the polynomial form as follows.

\begin{proposition}\label{ATZ}
Let $P\in\Z[x]$, $\deg{P} \geq 1$, be irreducible in $\Z[x]$. Then
$\Z[x] = \nn[x] + (P)$ holds for some finite set $\nn \subset \Z$ if and only if $P$ has no root $\alpha$ with $|\alpha|<1$.
\end{proposition}
\begin{proof}[Proof of Proposition \ref{ATZ}]
Since $P \in \Z[x]$ is not a constant, it has a root $\al \in \C$.  The irreducibility of $P$ in $\Z[x]$ implies that the g.c.d.\ of the coefficients of $P$ is $1$. Therefore, by Gauss Lemma (see, for instance, \cite{Anderson:00}), the kernel of the evaluation mapping $x \mapsto \al$ is the principal ideal $(P)$. Then, by the First Isomorphism Theorem, the ring $\Z[\al]$ (the image of a ring $\Z[x]$ under the evaluation mapping) satisfies $\Z[\al] \cong \Z[x]/(P)$, {\it i.e.}, each element of $Q(\al) \in \Z[\al]$ corresponds to a residue class $Q(x) + (P) \in \Z[x]/(P)$. In particular, it is readily seen that $\Z[\al]=\nn[\al]$ is equivalent to the finiteness property $\Z[x] = \nn[x]+(P)$. 
\end{proof}

Van de Woestijne and his coauthors~\cite{SSTW, Woe} consider products of digit systems. The following Lemma is simpler than their results in the sense that we do not care about the size of the digit set used in the product and we only think about integer digits. 

\begin{lemma}\label{lemMult}
Let $P,Q \in \Z[x]$ and let $\ll,\mm\subset\Z$ be finite sets. If  $\Z[x]$ satisfies $\Z[x] = \ll[x]+(P)$ and $\Z[x]=\mm[x]+(Q)$, then there exists a finite set $\nn\subset \Z$, such that $\Z[x]  = \nn[x]+(PQ)$.
\end{lemma}

\begin{proof}[Proof of Lemma \ref{lemMult}] 

Let $S$ be an arbitrary element of $\Z[x]$. By our assumptions we can choose a polynomial $R_1 \in \ll[x]$ in a way that $S = R_1 + S_1P$ holds for some $S_1  \in \Z[x]$. After that, we can choose a polynomial $R_2 \in \mm[x]$ such that, for some $S_2\in \Z[x]$, $S_1 = R_2 + S_2Q$, which yields
\[
S = R_1 + (R_2 + S_2Q)P = R_1 + R_2P + S_2PQ.
\]
This means $S \equiv R_1 + R_2 P  \pmod{PQ}$.
The coefficients of the polynomials $R_1,R_2$ are contained in finite sets $\ll$ and $\mm$, respectively. Since $P$ is a fixed polynomial, the set of  possible coefficients of the polynomial $R_1+R_2P$ is bounded and therefore is contained in a finite set $\nn$ depending only on $\ll$, $\mm$,  and the polynomial $P$. This implies that each $S\in \Z[x]$ has a representative in $\nn[x]$ modulo $PQ$.
\end{proof}

By induction, Lemma~\ref{lemMult} implies the following corollary.

\begin{corollary}\label{lemMult2}
Let $P \in \Z[x]$ and let $\nn\subset\Z$ be a finite set. If\, $\Z[x] = \nn[x] + (P)$, then, for every $m \in \N$, there exists a finite set $\nn_m \subset \Z$, such that $\Z[x]  = \nn_m[x] + (P^m)$.
\end{corollary}

Moreover, we gain the following generalization of Proposition~\ref{ATZ} to arbitrary polynomials over $\Z$. 

\begin{theorem}\label{ATZ2}
Let $P\in\Z[x]$. Then $\Z[x] = \nn[x] + (P)$ holds for some finite set $\nn \subset \Z$ if and only if $P$ has no root $\alpha$ with $|\alpha|<1$.
\end{theorem}
\begin{proof}[Proof of Theorem~\ref{ATZ2}] We begin by showing sufficiency. First, assume that $P$ is a nonzero constant polynomial, \emph{i.e.}, $P=c\in \Z$, $c \ne 0$. Let $\nn \subset \Z$ be a complete set of representatives of the residue classes $\pmod{c}$. It is clear that one can write any $S \in \Z[x]$ as $S = R + cQ$, $R \in \nn[x]$, $Q \in \Z[x]$, therefore the statement is true in this case. If $\deg P \ge 1$, then, by unique factorization in $\Z[x]$, one can express $P=\prod_{j=1}^{k}P_j^{m_j}$, where each $P_j \in \Z[x]$ is irreducible and has no root $\al \in \C$ with $\abs{\al} < 1$. By combining previous remark about constant polynomials, Proposition~\ref{ATZ}, Lemma~\ref{lemMult}, and Corollary~\ref{lemMult2}, we obtain that there exists a finite set $\nn \subset \Z$, such that $\Z[x]=\nn[x]+(P)$. For the necessity, suppose that $\Z[x]=\nn[x]+(P)$ holds. Let $Q \in \Z[x]$ be any irreducible divisor of $P \ne 0$. Since $(P) \subset (Q)$, $\Z[x]=\nn[x]+(Q)$. Then either $Q=c \ne 0$, or the `only if' part of Proposition \ref{ATZ} applies. In both cases, $Q$ has no root $\al$ with $\abs{\al}<1$.
\end{proof}

Recall that the companion matrix $C(Q)$ of a polynomial
\[
Q(x) = a_dx^d + a_{d-1}x^{d-1}+\dots + a_1x +a_0 \in \Z[x]
\] of degree $d := \deg P \geq 1$ is defined by
\begin{equation*}
C(Q) = \begin{pmatrix}
0 & 0 & \dots & 0& 0 & -a_0/a_d\\
1 & 0 & \dots & 0& 0 & -a_1/a_d\\
0 & 1 & \dots & 0& 0 & -a_2/a_d\\
\vdots & \vdots & \ddots &\vdots & \vdots & \vdots\\
0 & 0 & \dots & 1& 0 & -a_{d-2}/a_d\\
0 & 0 & \dots & 0& 1 & -a_{d-1}/a_d\\
\end{pmatrix}.
\end{equation*}
One can see easily that $Q=a_d \chi_{_C}(x)$, where $\chi_{_C}(x)$ denotes the characteristic polynomial of $C(Q)$. Before proving a full matrix version of Theorem \ref{ATZ2}, we state an intermediate result.

\begin{lemma}\label{mainIrred}
Let $P \in \Z[x]$, $d=\deg{P} \geq 1$, and let $C=C(P)$ be the companion matrix of $P$. If $P$ has no root $\al$ with $\abs{\al} < 1$, then there exists a finite set of vectors $\dd \subset \Z^d$ such that $(C,\dd)$ is a digit system with finiteness property.
\end{lemma}
\begin{proof}[Proof of Lemma \ref{mainIrred}]
In view of \eqref{defZA}, each  vector $\bz \in \Z^d[C]$ can be written as
\[
\bz = Q_1(C)\boe_1 + \dots + Q_d(C)\boe_d, \text{ for some } Q_j \in \Z[x], 1 \leq j \leq d.
\]
Therefore,
\[
\begin{split}
\Z^d[C] &=\Z[C]\boe_1+\Z[C]\boe_2+\dots+\Z[C]\boe_d=\\
&= \Z[C]\boe_1+C\Z[C]\boe_1+\dots+C^{d-1}\Z[C]\boe_1.
\end{split}
\]
By Theorem~\ref{ATZ2}, there exists a finite set $\nn\subset\Z$ such that 
\[
\Z[x]=\nn[x] +(P). 
\]
Since $P(C)=0$, the substitution $x\mapsto C$ yields $\Z[C]=\nn[C]$.
Hence,
\begin{equation}\label{eq:zcc}
\begin{split}
\Z^d[C] &=\nn[C]\boe_1+C\nn[C]\boe_1+\dots+C^{d-1}\nn[C]\boe_1=\\
&=(\nn[C]+C\nn[C]+\dots+C^{d-1}\nn[C])\boe_1.
\end{split}
\end{equation}
As $\nn \subset \Z$ is finite, one can always find a finite set $\mm \subset \Z$, such that
\[
\nn[x]+ x\nn[x] + \dots + x^{d-1}\nn[x] \subset \mm[x].
\]
Together with \eqref{eq:zcc} this yields
$\Z^d[C] \subset \mm[C]\boe_1$.
By setting $\dd := \mm\boe_1 \subset \Z^d$, we get $\Z^d[C] \subset \dd[C]$ and, {\it a fortiori}, $\Z^d[C] = \dd[C]$. Therefore, the digit system $(C,\dd)$ in $\Z^d[C]$ possesses the finiteness property.
\end{proof}

Following Kov{\'a}cs \cite{Kov8, Kov9}, we now define \emph{block-digit systems}. For two vectors $\bv \in \Q^m$, $\bw \in \Q^n$, their \emph{block sum} is defined by $\bv \oplus \bw := (\bv, \bw)^T \in \Q^{m+n}$. Likewise, the block-sum of two matrices $A \in M_m(\Q)$ and $B \in M_n(\Q)$ is a  \emph{block matrix}
\[
A \oplus B = \left(\begin{array}{ll} A & O_{m, n}\\
					  O_{n, m} &  B\\
		\end{array}\right) \in M_{m+n}(\Q).
\]
The block sum respects the addition of vectors in $\Q^{m+n}$ and the multiplication by compatible block-matrices from $M_{m+n}(\Q)$, in particular,
\begin{equation}\label{blockop}
\bv \oplus \bw + \bx \oplus \by = (\bv + \bx) \oplus (\bw + \by), \qquad (A \oplus B) (\bv \oplus \bw) = (A\bv) \oplus (B\bw),
\end{equation} 
for every $\bv, \bx \in \Q^m$, $\bw, \by \in \Q^n$.
Let $\dd_A\subset \Z^m$ and $\dd_B\subset \Z^n$ be finite sets containing the respective zero vectors $\bnull_m \in \Z^m$ and $\bnull_n \in \Z^n$. The presence of  $\bnull_m \in \dd_A$ and $\bnull_n \in \dd_B$ allows assembling the radix representations of different length to obtain
\begin{equation}\label{blocksumNS}
\dd_A[A] \oplus \dd_B[B] = (\dd_A \oplus \dd_B)[A \oplus B] 
\end{equation} using the block-sum properties in \eqref{blockop}. Properties \eqref{blockop} and \eqref{blocksumNS}, combined with the projections into the first $m$ or the last $n$ coordinates, implies the following proposition.
\begin{proposition}\label{propBlock}
Let $A \in M_m(\Q)$ and $B \in  M_n(\Q)$. Suppose that $\bnull_m \in \dd_A \subset \Z^m$ and $\bnull_n \subset \dd_B \in \Z^n$. Then $(A \oplus B, \dd_A \oplus \dd_B)$ is a digit system with finiteness property in the lattice $\Z^{m+n}[A \oplus B]$ if and only if digit systems $(A, \dd_A)$ and $(B, \dd_B)$ in $\Z^{m}[A]$ and $\Z^{n}[B]$, respectively, both have the finiteness property.
\end{proposition}
Finally we are ready to prove the main theorem of the paper.

\begin{proof}[Proof of Theorem \ref{thmMain}] First, following Gilbert \cite{Gi2} or Vince \cite{Vin2}, we will show that $(A, \dd)$ is never a digit system with finiteness property if $A$ has an eigenvalue $\la$ of absolute value $|\la| < 1$.
Indeed, assume that $\Z^n[A]=\dd[A]$, so that a $\bz \in \Z^n[A]$ can be written $\bz = \bd_0 + \dots + A^k \bd_k$, $k \in \N$, $\bd_j \in \dd$.
Let $\bv \ne 0 \in \C^n$ satisfy $\bv^TA = \la \bv^T$ with $|\la| < 1$. Then
\[
\bv^T \bz = \bv^T\bd_0 + \lambda \bv^T\bd_1 + \dots + \lambda^k \bv^T\bd_k.
\]
Since $\dd$ is finite, there exists $C > 0$, such that  $\abs{\bv^T\bd} < C$ for every $\bd \in \dd$. Hence, $\abs{\la} < 1$ yields
\[
\abs{\bv^T\bz} < C(1 + \abs{\la} + \dots + \abs{\la}^k)< C(1-\abs{\la})^{-1}.
\] This is not possible, since $\Z^n[A]$, in particular, contains the lattice $\Z^n$ and its image under the linear mapping $\bz \mapsto \bv^T \bz$ cannot be bounded for $\bv \ne \bnull$.

Thus, it suffices to consider the case when all the eigenvalues of $A$ are greater than or equal to $1$ in absolute value. By the Primary Decomposition Theorem (\cite[Chapter~XI, Theorem~4.1]{Lang1}) and the Cyclic Subspace Decomposition Theorem (\cite[Chapter~XIV, Theorem~2.1]{Lang2}), there exists a nondegenerate $n \times n$ rational matrix $T$, such that $A = TBT^{-1}$, where $B$ is the block-diagonal Frobenius form
\[
B = \left(\begin{array}{lll}
		B_1 & \dots       & O_{n_1, n_k}\\
		\vdots              & \ddots     & \vdots \\ 
		O_{n_k, n_1}                    & \dots  & B_k\\ 
	\end{array}\right).
\]
Each block $B_j := C(P_j^{m_j})$ of the size $n_j \times n_j$ on the main diagonal of $B$ is the companion matrix of the $m_j$-th power of the irreducible polynomial $P_j\in \Z[x]$ that divides the characteristic polynomial of $A$ in $\Q[x]$. For every integer $c \ne 0$, $TBT^{-1}=(cT)B(cT)^{-1}$, therefore we may assume that $T \in M_n(\Z)$.  Since $A$ has no eigenvalue $|\la| < 1$, Lemma \ref{mainIrred} yields that, for each $B_j$, there exists a finite set $\dd_j \in \Z^{n_j}$, such that $(B_j, \dd_j)$ is a digit system with finiteness property in the module $\Z^{n_j}[B_j]$. We may assume that $\bnull_{n_j} \in \dd_j$ (or append it to each digit set $\dd_j$). By Proposition \ref{propBlock}, the block-sum
\[
\dd' := \dd_1 \oplus \dots \oplus \dd_k
\] 
makes a digit set for the digit system  with finiteness property $(B, \dd')$ in the module $\Z^n[B]$. By conjugating with $T$, we obtain that $(A, T\dd')$ is a digit system  with finiteness property in the module $(T\Z^n)[A]$. Since $T\Z^n$ has finite index in $\Z^n$, by adding all coset representatives from $\dd'':=\Z^n/T\Z^n$ to the digit set $\dd'$, one obtains the digit system $(A, \dd' + \dd'')$ with finiteness property in the module $\Z^n[A]$.
\end{proof}

We end our paper by posing two open problems.

\begin{problem}\label{prob1}
Suppose an  $n \times n$ matrix $A \in M_n(\Q)$ with rational entries has eigenvalues $\la$ with $\abs{\la} \ge 1$, and that at least one eigenvalue is of absolute value $\abs{\la}=1$. Is it true that  a digit system $(A, \dd)$ in $\Z^n[A]$ cannot admit unique representations?   
\end{problem}

In relation to Problem~\ref{prob1}, Vince \cite{Vin2} proved (see \cite[Proposition~4]{Vin2}), that when the entries of $A$ are integers, there exists no digit system $(A,\dd)$ in $\Z^n$ with a zero digit $0 \in \dd$ satisfying the unique representation property. He notes that the eigenvalues $\la$ with $\abs{\la}=1$ of an integer matrix $A$ are roots of unity; thus, there must be a non-zero vector $\bv \in \Z^n$, such that $A^m \bv = \bv$ for some integer $m \in \N$. By uniqueness of the representation, $\bv$ has the radix representation with all digits equal to $\bnull$, which contradicts $\bv \ne \bnull$. However, in case $A \in M_n(\Q)$ has rational entries, eigenvalues of $A$ on the unit circle, in general, need not be only the roots of unity. Problem~\ref{prob1} seems to be  much more interesting in this setting. In this context we refer to \cite{MR1657992}, where an interesting characterization of the polynomials in $\R[x]$ having all their roots on the unit circle is given and to \cite{Brunotte-Kirschenhofer-Thuswaldner:12} where Problem~\ref{prob1} is addressed in the context of so-called \emph{shift radix systems}.

\begin{problem}\label{prob2}
Suppose that $A \in M_n(\Q)$ satisfies the assumptions of Problem \ref{prob1}, and that  $(A, \dd)$ in $\Z^n[A]$ has the finiteness property. What is the smallest possible size $\#\dd$ of the digit set?
\end{problem}

Regarding Problem \ref{prob2}, we know that in the case where the characteristic polynomial of $A$ is irreducible over $\Q$ and all the eigenvalues are of absolute value $1$, the construction in the proof of Theorem~\ref{ATZ} provided in \cite{AkThZa} yields only a very rough bound for $\#\dd$ that is surely far from the optimum.

\bibliographystyle{abbrv}
\bibliography{eigen1}

\end{document}